\documentclass[11pt]{article}
\usepackage{amsmath, amssymb, amsfonts, amsthm, amscd, vmargin}
\usepackage[latin1]{inputenc}
\usepackage[all]{xy}
\usepackage{graphicx}
\usepackage{enumerate}

\setmarginsrb{2.4cm}{4cm}{2.4cm}{3.3cm}{0cm}{0cm}{0cm}{1.5cm}

\theoremstyle{plain}
\newtheorem{teo}{Theorem}[section]
\newtheorem{lem}[teo]{Lemma}
\newtheorem{prop}[teo]{Proposition}
\newtheorem{cor}[teo]{Corollary}

\theoremstyle{definition}
\newtheorem{defi}[teo]{Definition}
\newtheorem{ex}[teo]{Example}

\newtheorem{rems}[teo]{Remarks}

\numberwithin{equation}{teo}

\newcommand{\Cbb}{{\mathbb C}}

\newcommand{\Zbb}{{\mathbb Z}}

\newcommand{\Pbb}{{\mathbb P}}

\newcommand{\Gbb}{{\mathbb G}}

\newcommand{\lra}{\longrightarrow}
\newcommand{\lmt}{\longmapsto}

\begin{document}

\title{Vanishing theorem for the cohomology of line bundles on Bott-Samelson varieties}

\author{Boris Pasquier}
\maketitle

\begin{abstract}
We use the toric degeneration of Bott-Samelson varieties and the description of cohomolgy of line bundles on toric varieties to deduce vanishings results for the cohomology of lines bundles on Bott-Samelson varieties. 
\end{abstract}




\section*{Introduction}

Bott-Samelson varieties were originally defined as desingularizations of Schubert varieties and were used to describe the geometry of Schubert varieties. In particular, the cohomology of some line bundles on Bott-Samelson varieties were used to prove that Schubert varieties are normal, Cohen-Macaulay and with rational singularities (see for example \cite{BK05}). In this paper, we will be interested in the cohomology of all line bundles of Bott-Samelson varieties.

We consider a Bott-Samelson variety  $X(\tilde{w})$  over an algebraically closed field $\Bbbk$  associated to an expression $\tilde{w}=s_{\beta_1}\dots s_{\beta_N}$ of an element $w$ in the Weyl group of a Kac-Moody group $G$ over $\Bbbk$ (see Definition \ref{BS} (i)).

In the case where $G$ is semi-simple,  N.~Lauritzen and J.F.~Thomsen proved, using Frobenius splitting, the vanishing of the cohomology in positive degree of line bundles on $X(\tilde{w})$ of the form $\mathcal{L}(-D)$ where $\mathcal{L}$ is any globally generated line bundle on $X(\tilde{w})$ and $D$ a subdivisor of the boundary of $X(\tilde{w})$ corresponding to a reduced expression of $w$ \cite[Th7.4]{LT04}.  The aim of this paper is to give the vanishing in some degrees of the cohomology of any line bundles on $X(\tilde{w})$.\\

Let us define, for all $\epsilon=(\epsilon_k)_{k\in\{1,\dots,N\}}\in\{+,-\}^N$ and for all integers $1\leq i<j\leq N$, $$\alpha_{ij}^\epsilon:=\langle\beta_i^\vee,(\prod_{i<k<j,\,\epsilon_k=-}s_{\beta_k})(\beta_j)\rangle.$$
These integers are natural geometric invariants of the Bott-Samelson variety, they also appear, for example, in \cite[Theorem 3.21]{Wi06} in product formula in the equivariant cohomology of complex Bott-Samelson varieties .

Since $X(\tilde{w})$ is smooth, we can consider divisors instead of line bundles. Thus, let us denote by $Z_1,\dots,Z_N$ the natural basis of divisors of $X(\tilde{w})$  (see Definition \ref{BS} (ii)).
Let $D:=\sum_{i=1}^Na_iZ_i$ be any divisor of $X(\tilde{w})$.  

Let $i\in\{1,\dots,N\}$. We say that $D$ satisfies condition $(C_i^+)$ if for all $\epsilon\in\{+,-\}^N$, we have $$C_i^\epsilon:=a_i+\sum_{j>i,\,\epsilon_j=+}\alpha_{ij}^\epsilon a_j\geq -1$$ and we say that $D$ satisfies condition $(C_i^-)$ if for all $\epsilon\in\{+,-\}^N$, we have $$C_i^\epsilon:=a_i+\sum_{j>i,\,\epsilon_j=+}\alpha_{ij}^\epsilon a_j\leq -1.$$

The main result of this paper is the following.

\begin{teo}\label{main}
Let $X(\tilde{w})$ be a Bott-Samelson variety and $D$ a divisor of $X(\tilde{w})$. Let $\eta\in\{+,-,0\}^N$. Define two integers $\eta^+:=\sharp\{1\leq j \leq N\mid\eta_j=+\}$ and $\eta^-:=\sharp\{1\leq j \leq N\mid\eta_j=-\}.$

Suppose that $D$ satisfies conditions $(C_i^{\eta_i})$ for all $i\in\{1,\dots,N\}$ such that $\eta_i\neq 0$.

Then $H^i(X,D)=0,\mbox{ for all }i<\eta^-\mbox{ and for all }i>N-\eta^+.$
\end{teo}

Let us remark that, Conditions $(C_N^+)$ and $(C_N^-)$ are respectively $a_N\geq -1$ and $a_N\leq -1$, so that $\eta_N$ can always be chosen different from $0$. Thus, for any divisor $D$ of $X(\tilde{w})$,  Theorem \ref{main} gives the vanishing of the cohomology of $D$ in at least one degree.

Although Theorem \ref{main} gives us a lot of cases of vanishing, it does not permit to recover all the result of N.~Lauritzen and J.P.~Thomsen. See Example \ref{exemple} to illustrate this facts.

However, for  lots of divisors, Theorem \ref{main} gives the vanishing of their cohomology in all degrees except one. More precisely, we have the following.
\begin{cor}\label{cornonvide}
Let $D=\sum_{i=1}^Na_iZ_i$ be a divisor of $X=X(\tilde{w})$. Suppose that, for all $i\in\{1,\dots,N\}$, one of the following two conditions $\tilde{C}_i^+$ and $\tilde{C}_i^-$ is satisfied:
$$\begin{array}{cc}
\tilde{C}_i^+: & a_i\geq -1+\max_{\epsilon\in\{+,-\}^N}(-\sum_{j>i,\,\epsilon_j=+}\alpha_{ij}^\epsilon a_j) \\
\tilde{C}_i^-: & a_i\leq -1+\min_{\epsilon\in\{+,-\}^N}(-\sum_{j>i,\,\epsilon_j=+}\alpha_{ij}^\epsilon a_j)
\end{array}.$$
Then, $H^i(X,D)=0$ for all $i\neq\sharp\{1\leq j\leq N\mid\tilde{C}_j^-\mbox{ is satisfied }\}$.
\end{cor}

Let us remark that, for all $\eta\in\{+,-\}^N$, the set of points $(a_i)\in\Zbb^N$ satisfying $\tilde{C}_j^{\eta_j}$ for all $j\in\{1,\dots,N\}$ is a non empty cone. So that Corollary \ref{cornonvide} can be applied to infinitly  many divisors.\\

The strategy of the proof of Theorem \ref{main} is the following. In Section \ref{1}, we define and describe a family of deformation with general fibers the Bott-Samelson variety and with special fiber a toric variety. The toric variety we obtain is a Bott tower, its fan has a simple and well understood structure (for example it has $2N$ cones of dimension~1 and $2^N$ cones of dimension $N$). In Section \ref{2}, we describe how to compute the cohomology of divisors on the special fiber and we prove the same vanishings as in Theorem \ref{main} but for divisors on this toric variety. 
Then Theorem \ref{main} is a direct consequence of the semicontinuity Theorem \cite[III 12.8]{Ha77}.

\section{Toric degeneration of Bott-Samelson varieties}\label{1}

In this section we rewrite the theory of M.~Grossberg and Y.~Karshon \cite{GK94} on Bott-towers, in the case of Bott-Samelson varieties and in an algebraic point of view.

Let $A=(a_{ij})_{1\leq i,j\leq n}$ be a generalized Cartan matrix, {\it i.e.} such that (for all $i,j$) $a_{ii}=2$, $a_{ij}\leq 0$ for $i\neq j$, and $a_{ij}=0$ if $a_{ji}=0$.
Let $G$ be the ``maximal'' Kac-Moody group over $\Bbbk$ associated to $A$ constructed in \cite[Section 6.1]{Ku02} (see \cite{Ti81a} and \cite{Ti81b} in arbitrary characteristic). Note that, in the finite case, $G$ is the simply-connected semisimple algebraic group over $\Bbbk$.  Denote by $B$ the standard Borel subgroup of $G$ containing the standard maximal torus $T$.
Let $\alpha_1,\dots,\alpha_n$ be the simple roots of $(G,B,T)$ and $s_{\alpha_1},\dots,s_{\alpha_n}$ the associated simple reflections generating the Weyl Group $W$.
For all $i\in\{1,\dots,n\}$, denote by $P_{\alpha_i}:=B\cup Bs_{\alpha_i}B$ the minimal parabolic subgroup containing $B$ associated to $\alpha_i$.
Let $w\in W$ and $\tilde{w}:=s_{\beta_1}\dots s_{\beta_N}$ be an expression (not neccessarily reduced) of $w$, with $\beta_1,\dots,\beta_N$ simple roots.
For all $i$ and $j$ in $\{1,\dots,N\}$, denote by $\beta_{ij}$ the integer $\langle\beta_i^\vee,\beta_j\rangle$.

\begin{defi}\label{BS}
\begin{enumerate}[(i)]
\item The Bott-Samelson variety associated to $\tilde{w}$ is $$X(\tilde{w}):=P_{\beta_1}\times^B\dots\times^B P_{\beta_N}/B$$ where the action of $B^N$ on $P_{\beta_1}\times\dots\times P_{\beta_N}$ is defined by $$(p_1,\dots,p_N).(b_1,\dots,b_N)=(p_1b_1,b_1^{-1}p_2b_2,\dots,b_{N-1}^{-1}p_Nb_N),\,\forall p_i\in P_{\beta_i},\,\forall b_i\in B.$$
\item For all $i\in\{1,\dots,N\}$, we denote by $Z_i$ the divisor of $X(\tilde{w})$ defined by $\{(p_1,\dots,p_N)\in X(\tilde{w})\mid p_i\in B\}$.
Thus $(Z_i)_{i\in\{1,\dots,N\}}$ is a basis of the Picard group of $X(\tilde{w})$, and if $\tilde{w}$ is reduced it is the basis of effective divisor \cite[Section 3]{LT04}.
\end{enumerate}
\end{defi}

In order to define a toric degeneration of a Bott-Samelson variety, we need to introduce particular endomorphisms of $G$ and $B$.

Since the simple roots are linearly independant elements in the character group of $G$, one can choose a positive integer $q$ and an injective morphism $\lambda:\Bbbk^*\lra T$ such that for all $i\in\{1,\dots,n\}$ and all $u\in \Bbbk^*$, $\alpha_i(\lambda(u))=u^q$.
And let us define, for all $u\in \Bbbk^*$, $$\begin{array}{cccc}
\tilde{\psi_u} : & G & \lra & G \\
 & g & \lmt & \lambda(u)g\lambda(u)^{-1}.
\end{array}$$

The morphism $\psi$ from $\Bbbk^*$ to the set of endomorphism of $B$ defined by $\psi(u)=\tilde{\psi_u}_{|B}$ can be continuously extended to $0$. Indeed, the unipotent radical $U$ of $B$ lives in a group (denoted by $U^{(1)}$ in \cite{Ti81b}) where the action of $t\in T$ by conjugation is, on some generators (except the identity), the multiplication by some positive powers of $\alpha_i(t)$ for some $i\in\{1,\dots,n\}$. Then for all $x\in U$, $\psi(u)$ goes to the identity when $u$ goes to zero.

We denote, for all $u\in \Bbbk$, by $\psi_u$ the morphism $\psi(u)$. Remark that $\psi_0$ is the projection from $B$ to $T$.

We are now able to give the following 
\begin{defi}
\begin{enumerate}[(i)]
\item Let $\mathfrak{X}\lra \Bbbk$ be the variety defined by $$\mathfrak{X}:=\Bbbk\times P_{\beta_1}\times\dots\times P_{\beta_N}/B^N$$  where the action of $B^N$ on $\Bbbk\times P_{\beta_1}\times\dots\times P_{\beta_N}$ is defined by $\forall u\in \Bbbk,\,\forall p_i\in P_{\beta_i},\,\forall b_i\in B$, $$(u,p_1,\dots,p_N).(b_1,\dots,b_N)=(u,p_1b_1,\psi_u(b_1)^{-1}p_2b_2,\dots,\psi_u(b_{N-1})^{-1}p_Nb_N).$$
\item  For all $i\in\{1,\dots,N\}$, we denote by $\mathcal{Z}_i$ the divisor of $\mathfrak{X}$ defined by $$\{(u,p_1,\dots,p_N)\in \mathfrak{X}\mid p_i\in B\}.$$
\end{enumerate}
\end{defi}

For all $u\in \Bbbk$, we denote by $\mathfrak{X}(u)$ the fiber of $\mathfrak{X}\lra \Bbbk$ over $u$.

\begin{prop}
\begin{enumerate}[(i)] \item For all $u\in \Bbbk^*$, $\mathfrak{X}(u)$ is isomorphic to the Bott-samelson variety $X(\tilde{w})$ such that, for all $i\in\{1,\dots,N\}$, the divisor $\mathcal{Z}_i(u):=\mathfrak{X}(u)\cap\mathcal{Z}_i$  corresponds to the divisor $Z_i$ of $X(\tilde{w})$.
\item $\mathfrak{X}(0)$ is a toric variety of dimension $N$.
\end{enumerate}
\end{prop}

\begin{proof}
\begin{enumerate}[(i)] \item  Remark first that $\mathfrak{X}(1)$ is by definition the Bott-Samelson variety and, that for all $i\in\{1,\dots,N\}$, $\mathcal{Z}_i(1)=Z_i$  . Now let $u\in \Bbbk^*$ and check that $$\begin{array}{cccl}
\theta_u : & \mathfrak{X}(1) & \lra & \mathfrak{X}(u) \\
 & (p_1,\dots,p_N) & \lmt & (p_1,\tilde{\psi_u}(p_2),\tilde{\psi_u}^2(p_3),\dots,\tilde{\psi_u}^{N-1}(p_N)).
\end{array}$$ is well-defined and is an isomorphism. Moreover, for all $i\in\{1,\dots,N\}$, $p_i$ is in $B$ if and only if $\tilde{\psi_u}(p_i)$ is in $B$, so that $\theta_u(Z_i)=\mathcal{Z}_i(u)$.
\item  Let $T_{\beta_i}$ be the maximal subtorus of $T$ acting trivially on $P_{\beta_i}/B\simeq \Pbb^1_\Bbbk$. Now, since $\psi_0(b)$ commutes with $T$ for all $b\in B$, one can define an effective action of $\prod_{i=1}^NT/T_{\beta_i}\simeq (\Bbbk^*)^N$ on $\mathfrak{X}(0)$ as follows
$$\forall t_i\in T,\,\forall p_i\in P_{\beta_i},\,(t_1,\dots,t_N).(p_1,\dots,p_N)=(t_1p_1t_1^{-1},t_2p_2t_2^{-1},\dots,t_Np_Nt_N^{-1}).$$
Moreover, since $T/T_{\beta_i}\simeq\Bbbk^*$ acts on $P_{\beta_i}/B\simeq\Pbb^1_\Bbbk$ with an open orbit, $(\Bbbk^*)^N$ acts also with an open orbit in $\mathfrak{X}(0)$.
\end{enumerate}
\end{proof}

\begin{prop}\label{fan}
Let $(e_1^+,\dots,e_N^+)$ be a basis of $\Zbb^N$. Define, for all $i\in\{1,\dots,N\}$, the vector $e_i^-:=-e_i^+-\sum_{j>i}\beta_{ij}e_j^+$.

Then a fan $\mathbb{F}$ of $\mathfrak{X}(0)$ consists of cones generated by subsets of $\{e_1^+,\dots,e_N^+,e_1^-,\dots,e_N^-\}$ containing no subset of the form $\{e_i^+,e_i^-\}$. (In other words,the fan whose maximal cones are the cones generated by $e_1^{\epsilon_1},\dots,e_N^{\epsilon_N}$ with $\epsilon\in\{+,-\}^N$.)

Moreover,  for all $i\in\{1,\dots,N\}$, $\mathcal{Z}_i(0)$ is the irreducible $(\Bbbk^*)^N$-stable divisor of $\mathfrak{X}(0)$ corresponding to the one dimensional cone of $\mathbb{F}$ generated by $e_i^+$.
\end{prop}

\begin{ex}\label{exemple1}
If $G=\operatorname{SL(3)}$ and $\tilde{w}=s_{\alpha_1}s_{\alpha_1}$, we have the following fan.

\begin{center}
\includegraphics[width=5cm]{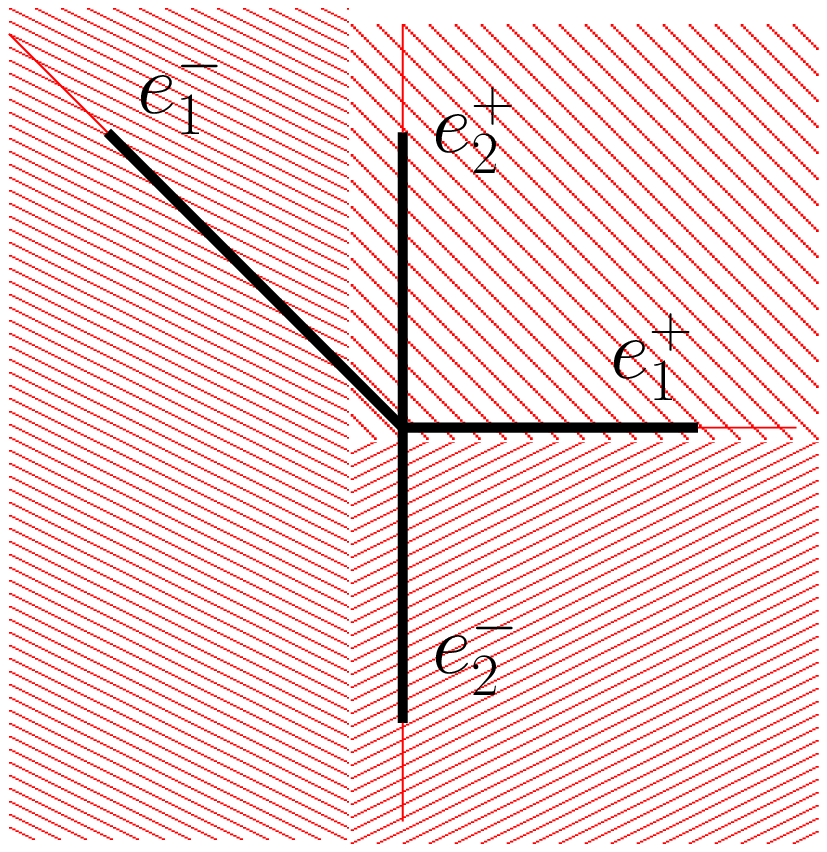}
\end{center}
\end{ex}

In fact, one can prove that the Bott-Samelson variety of Example \ref{exemple1} is isomorphic the toric variety $\mathfrak{X}(0)$. But this is not the general case. For example, if $G=\operatorname{SL(2)}$ and $\tilde{w}=s_{\alpha_1}s_{\alpha_2}$, $X(\tilde{w})$ is a toric variety (it is in fact $\Pbb^1_\Bbbk\times\Pbb^1_\Bbbk$) but it is not isomorphic to $\mathfrak{X}(0)$. And, if $G=\operatorname{SL(4)}$ and $\tilde{w}=s_{\alpha_2}s_{\alpha_1}s_{\alpha_3}s_{\alpha_2}$, then $X(\tilde{w})$ is not a toric variety.

\begin{proof}
Let us first write a few technical results.
For all simple roots $\alpha$, there exists a unique closed subgroup $\mathcal{U}_\alpha$ of $G$ and an isomorphism $$u_\alpha: \Gbb_a\lra \mathcal{U}_\alpha\mbox{ such that } \forall t\in T,\,\forall x\in \Bbbk,\, tu_\alpha(x)t^{-1}=u_\alpha(\alpha(t)x).$$
Moreover, the $u_\alpha$ can be chosen such that $n_\alpha:=u_\alpha(1)u_{-\alpha}(-1)u_\alpha(1)$ is in the normalizer of $T$ in $G$ and has image $s_\alpha$ in $W$. And for all $x\in\Bbbk^*$ we have $$u_{-\alpha}(x)u_{\alpha}(-x^{-1})u_{-\alpha}(x)=\alpha^\vee(x^{-1})n_{-\alpha}.$$
See \cite[Chapter8]{Sp98} for the finite case. And we can reduce to the finite case in the general case by construction of $G$.

Then for all $x\in\Bbbk^*$ we also have \begin{equation}\label{formule1} n_{-\alpha} u_{-\alpha}(-x)=\alpha^\vee(x)u_{-\alpha}(x)u_\alpha(-x^{-1})=u_{-\alpha}(x^{-1})\alpha^\vee(x)u_\alpha(-x^{-1}).\end{equation}
Remark also that, for all simple root $\alpha$, the subgoup $\mathcal{U}_{-\alpha}$ is a subgroup of $P_\alpha$ and $n_{-\alpha}\in P_\alpha$.

Then, for all $\epsilon\in\{0,1\}^N$, we can define an embedding $\phi_\epsilon$ of $\Bbbk^N$ in $\mathfrak{X}(0)$ by $$\phi_\epsilon(x_1,\cdots,x_N)=((n_{-\beta_1})^{\epsilon_1}u_{-\beta_1}((-1)^{\epsilon_1})x_1),\dots,(n_{-\beta_N})^{\epsilon_N}u_{-\beta_N}((-1)^{\epsilon_N})x_N).$$ Note that, the $\phi_\epsilon(\Bbbk^N)$ with $\epsilon\in\{0,1\}^N$, are the  maximal affine $(\Cbb^*)^N$-stable subvarieties of $\mathfrak{X}(0)$. 
Moreover, if for all $i\in\{1,\dots,N\}$, $x_i\in\Bbbk^*$,  we prove by induction, using Equation \ref{formule1} and the defintion of $\mathfrak{X}(0)$, that \begin{eqnarray} \phi_\epsilon(x_1,\cdots,x_N) & = & (u_{-\beta_1}(x_1^{(-1)^{\epsilon_1}})\beta_1^\vee(x_1^{-1})^{\epsilon_1},\dots,u_{-\beta_N}(x_N^{(-1)^{\epsilon_N}})\beta_N^\vee(x_N^{-1})^{\epsilon_N})\nonumber\\
& = & (u_{-\beta_1}(x_1^{(-1)^{\epsilon_1}}), \dots, u_{-\beta_i}(x_i^{(-1)^{\epsilon_i}}\prod_{j<i}x_j^{-\epsilon_j\beta_{ji}}),\dots).\label{formule2}\end{eqnarray}

Now, let us compute the weight of regular functions of all these affine subvarieties. We need first to fix a basis of characters of $(\Bbbk^*)^N$.  Let us denote, for all $i\in\{1,\dots,N\}$, by $X_i$ the function in $\Bbbk(\mathfrak{X}(0))=\Bbbk(\phi_{(0,\dots,0)}((\Gbb_a)^N))$ defined by $X_i((u_{-\beta_1}(x_1),\dots,u_{-\beta_N}(x_N))=x_i$. Denote also by $(\chi_i)_{i\in\{1,\dots,N\}}$ the weights with $(\Bbbk^*)^N$ acts on $(X_i)_{i\in\{1,\dots,N\}}$, and by $(e_i^+)_{i\in\{1,\dots,N\}}$ the dual basis of $(\chi_i)_{i\in\{1,\dots,N\}}$. 
Then, if $\chi=\sum_{i=1}^Nk_i\chi_i$, we can check, using Equation \ref{formule2}, that $$\prod_{i=1}^NX_i^{k_i}\in \Bbbk[\phi_\epsilon((\Gbb_a)^N))]\Longleftrightarrow \forall i\in\{1,\dots,N\},\,\left\{\begin{array}{ll}
k_i=\langle\chi,e_i^+\rangle\geq 0 & \mbox{if }\epsilon_i=0 \\
-k_i-\sum_{j>i}\beta_{ij}k_j=\langle\chi,e_i^-\rangle\geq 0 & \mbox{if }\epsilon_i=1 \end{array}\right..$$
In other words, the cone associated to $\phi_\epsilon(\Bbbk^N)$ is generated by $e_1^{\epsilon_1'},\dots,e_N^{\epsilon_N'}$, where $\epsilon_i'=+$ and $-$ if $\epsilon_i=0$ and $1$ respectively.
It proves the first result of the proposition.

For the last statement, just remark that $\mathcal{Z}_i(0)$ is the divisor of $\mathfrak{X}(0)$ defined by the equation $X_i=0$, and that $X_i$ has weight $\chi_i$ which is the dual of $e_i^+$.
\end{proof}

\section{Cohomology of divisors on the toric variety $\mathfrak{X}(0)$}\label{2}

Let us first recall the result of M.~Demazure \cite{De70} on the cohomology of line bundles on smooth toric varieties. For the general theory of toric varieties, see \cite{Od88} or \cite{Fu93}.

Let $X$ be a smooth complete toric variety of dimension $N$ associated to a complete fan $\mathbb{F}$. Let $\Delta(1)$ be the set of primitive elements of one-dimensional cones of $\mathbb{F}$.
For all $\rho\in\Delta(1)$, we denote by $D_\rho$ the corresponding irreducible $(\Bbbk^*)^N$-stable divisor of $X$.
Let $D:=\sum_{\rho\in\Delta(1)}a_\rho D_\rho$. Let $h_D$ be the piecewise linear function associated to $D$, {\it i.e.} if $\mathcal{C}$ is the cone generated by $\rho_1,\dots,\rho_N$ then $h_{D|\mathcal{C}}$ is the linear function which takes values $a_{\rho_i}$ at $\rho_i$.

Denote by $X((\Bbbk^*)^N)$ be the set of characters of $(\Bbbk^*)^N$. For all  $m\in X((\Bbbk^*)^N)$, define the piecewise linear function $\phi_m:n\lmt \langle m,n\rangle+h_D(n)$. Let $\Delta(1)_m:=\{\rho\in\Delta(1)\mid\phi_m(\rho)<0\}.$ And define the simplicial scheme $\Sigma_m$ to be the set of all subset of $\Delta(1)_m$ generating a cone of $\mathbb{F}$ (we refer to \cite[Chapter I.3]{Go58} for cohomology of simplicial schemes).

The cohomology spaces $H^i(X,D)$ is a $(\Bbbk^*)^N$-module so that we have the following decomposition $$H^i(X,D)=\bigoplus_{m\in X((\Bbbk^*)^N)}H^i(X,D)_m.$$

M.~Demazure proved the following result.

\begin{teo}[\cite{De70}]\label{toriccohom}
With the notation above,
\begin{enumerate}[(i)]
\item if $\Sigma_m=\emptyset$, then  $H^0(X,D)_m=\Bbbk$ and $H^i(X,D)_m=0$ for all $i>0$;
\item if $\Sigma_m\neq\emptyset$, then $H^0(X,D)_m=0$, $H^1(X,D)=H^0(\Sigma_m,\Bbbk)/\Bbbk$ and $H^i(X,D)_m=H^{i-1}(\Sigma_m,\Bbbk)$ for all $i>1$.
\end{enumerate}
\end{teo}

Applying Theorem \ref{toriccohom} to $\mathfrak{X}(0)$, with the notation of the first section, one can deduce the following (the proof is left to the reader).

\begin{cor}\label{cortoriccohom} Let $\mathcal{D}=\sum_{i=1}^Na_i\mathcal{Z}_i$ be a divisor of $\mathfrak{X}$ and $\mathcal{D}(0)$ be the corresponding divisor $\sum_{i=1}^Na_i\mathcal{Z}_i(0)$ of $\mathfrak{X}(0)$.
\begin{enumerate}[(i)]
\item If there is an integer $j$ such that $\phi_m(e_j^+)\geq 0$ and $\phi_m(e_j^-)< 0$, or, $\phi_m(e_j^+)< 0$ and $\phi_m(e_j^-)\geq 0$, then $H^i(\mathfrak{X}(0),\mathcal{D}(0))_m=0$ for all $i\geq 0$.
\item If the condition above is not satisfied, let $j_m:=\sharp\{i\in\{1,\dots,N\}\mid\phi_m(e_j^+)<0\}$, then $H^i(\mathfrak{X}(0),\mathcal{D}(0))_m=0$ for all $i\neq j)m$ and $H^{j_m}(\mathfrak{X}(0),\mathcal{D}(0))_m=\Bbbk$.
\end{enumerate}
\end{cor}

\begin{ex}
If $G=\operatorname{SL(3)}$ and $\tilde{w}=s_{\alpha_1}s_{\alpha_2}$, if the simplical scheme $\Sigma_m$ is not empty, it is one of the following modulo symmetries.
\begin{center}
\includegraphics[width=12cm]{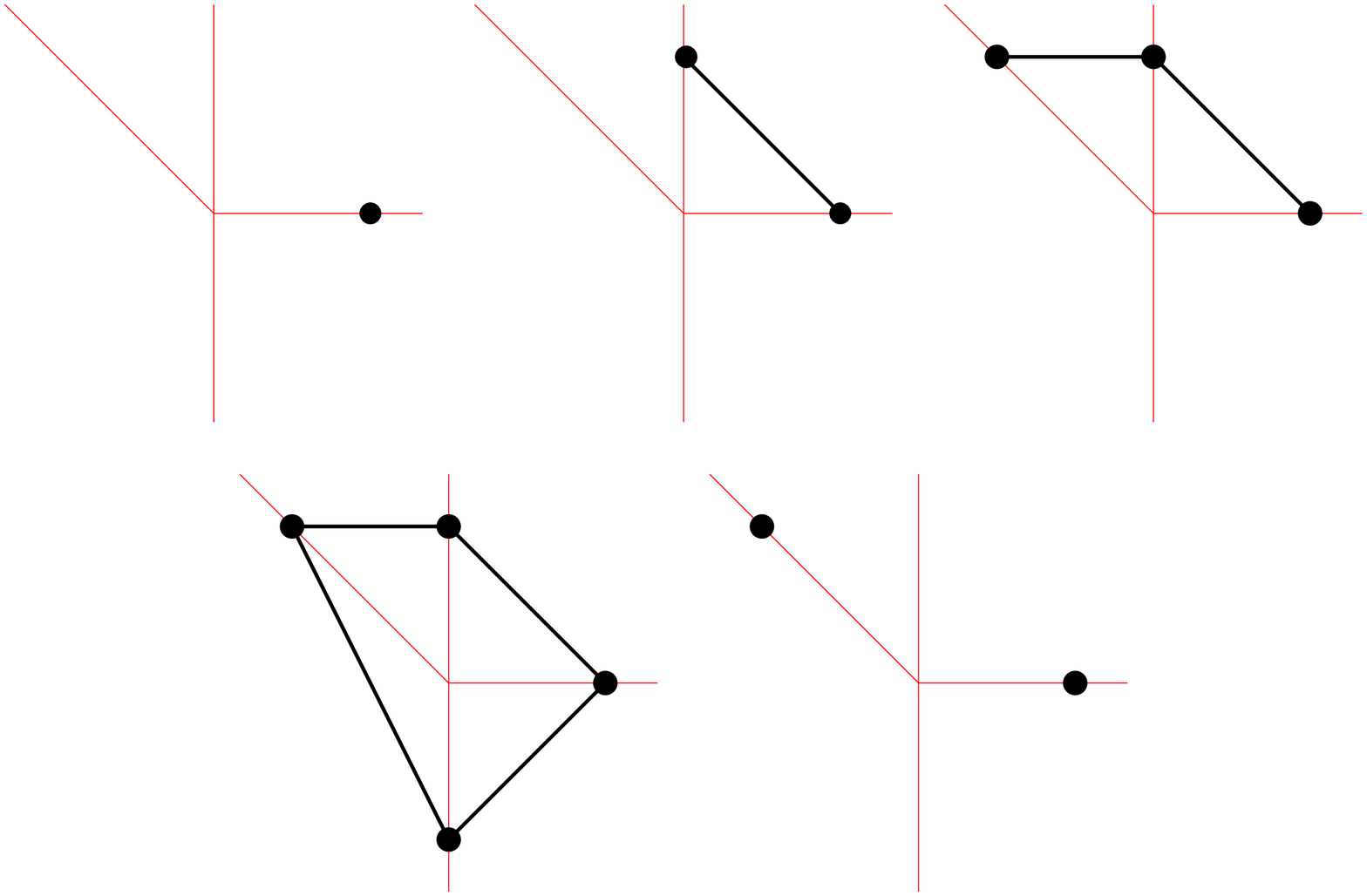}
\end{center}
In the first three cases, $H^0(\Sigma_m,\Bbbk)=\Bbbk$ and the cohomology of $\Sigma_m$ in positive degrees vanishes, and we are in the case (i) of Corollary \ref{cortoriccohom}. In fourth and fifth cases, we are in the case (ii) of Corollary \ref{cortoriccohom}. In fourth case, the non trivial cohomology are  $H^0(\Sigma_m,\Bbbk)=\Bbbk^2$ and $H^1(\Sigma_m,\Bbbk)=\Bbbk$. In fifth case, the only non trivial cohomology is $H^0(\Sigma_m,\Bbbk)=\Bbbk^2$.
\end{ex}

\begin{proof}[Sketch of proof of Corollary \ref{cortoriccohom}]
\begin{enumerate}[(i)]
\item Suppose $\phi_m(e_j^+)\geq 0$ and $\phi_m(e_j^-)< 0$. Then, all maximal simplices of $\Sigma_m$ contain $e_j^-$, so that $\Sigma_m$ is contractible.
\item One can check that $\Sigma_m$ is the set of faces of a $j_m$-dimensional convexe polytope.
\end{enumerate}
\end{proof}

We will now prove two lemmas. In the first one, we give a necessary condition on $m\in X((\Bbbk^*)^N)$ to satisfy the condition of Corollary \ref{cortoriccohom} (ii). The second lemma will be used to compute, in Case (ii), the possible values of $j_m$ which depend on the Conditions $(C_i^\pm)$.\\

First, for all $\epsilon\in\{+,-\}^N$, we define $x^\epsilon\in\Zbb^N$  by induction, as follows:  

$$x^\epsilon_i= \left\{\begin{array}{ll}
 -a_i & \mbox{ if }\epsilon_i=+\\
 -\sum_{j>i}\beta_{ij}x^\epsilon_j & \mbox{ if }\epsilon_i=-
 \end{array}\right.$$

\begin{lem}\label{convexhull}
Let $m\in\Zbb^N$ such that for all $i\in\{1,\dots,N\}$, we have either $\phi_m(e_i^+)\geq 0$ and $\phi_m(e_i^-)\geq 0$, or, $\phi_m(e_i^+)< 0$ and $\phi_m(e_i^-)< 0$.

Then $m$ is in the convex hull of the $x^\epsilon$.
\end{lem}

\begin{proof}
Since, for all $i\in\{1,\dots,N\}$,  $\phi_m(e_i^+)=m_i+a_i$ and $-\phi_m(e_i^-)=m_i+\sum_{j>i}\beta_{ij}m_j$ have opposite signs, there exists $N$ real numbers $\lambda_1,\dots,\lambda_N$ in $[0,1]$ such that, for all $i\in\{1,\dots,N\}$, $m_i=-\lambda_ia_i-(1-\lambda_i)\sum_{j>i}\beta_{ij}m_j$.
Denote, for all $\epsilon\in\{+,-\}^N$, by $m^\epsilon$ the product $$\prod_{\begin{array}{c}\scriptstyle 1\leq i\leq N\\ \scriptstyle \epsilon_i=+\end{array}}\lambda_i\times\prod_{\begin{array}{c}\scriptstyle 1\leq i\leq N\\ \scriptstyle \epsilon_i=-\end{array}}(1-\lambda_i).$$ Remark that $m^\epsilon\in [0,1]$.

Let us prove by induction that for all $i\in\{1,\dots,N\}$ we have $m_i=\sum_{\epsilon\in\{+,-\}^N}m^\epsilon x^\epsilon_i$, {\it i.e.} $m=\sum_{\epsilon\in\{+,-\}^N}m^\epsilon x^\epsilon$.

We will use the following easy fact: for all $i\in\{1,\dots,N\}$, \begin{equation}\lambda_i=\sum_{\begin{array}{c}\scriptstyle \epsilon\in\{+,-\}^N\\ \scriptstyle \epsilon_i=+\end{array}}\label{lambdai}
m^\epsilon.\end{equation} In particular, for $i=N$, we deduce, with the definition of $x_N^\epsilon$, that $\sum_{\epsilon\in\{+,-\}^N}m^\epsilon x^\epsilon_N=-\lambda_Na_N=m_N$.

Now let $i<N$ such that, for all $j>i$, $m_j=\sum_{\epsilon\in\{+,-\}^N}m^\epsilon x^\epsilon_j$. Then \begin{eqnarray*} m_i & = & -\lambda_ia_i-(1-\lambda_i)\sum_{j>i}\beta_{ij}m_j\\
 & = & -\lambda_ia_i-(1-\lambda_i)\sum_{j>i}\beta_{ij}\sum_{\epsilon\in\{+,-\}^N}m^\epsilon x^\epsilon_j\\
 & = & -\lambda_ia_i-(1-\lambda_i)\sum_{\epsilon\in\{+,-\}^N}m^\epsilon \sum_{j>i}\beta_{ij} x^\epsilon_j.
\end{eqnarray*}
Moreover, if for all $\epsilon\in\{+,-\}^N$, we define $\epsilon'\in\{+,-\}^N$ by $\epsilon'_j=\epsilon_j$ for all $j\neq i$ and $\epsilon'_i=-\epsilon_i$, we have $$\sum_{j>i}\beta_{ij} x^\epsilon_j= \left\{\begin{array}{ll}
 -x^\epsilon_i & \mbox{ if }\epsilon_i=-\\
 -x^{\epsilon'}_i & \mbox{ if }\epsilon_i=+
 \end{array}\right.$$

Then $$m_i= -\lambda_ia_i-(1-\lambda_i)\sum_{\begin{array}{c}\scriptstyle \epsilon\in\{+,-\}^N\\ \scriptstyle \epsilon_i=-\end{array}}(m^\epsilon+m
^{\epsilon'})x^\epsilon_i.$$

We conclude by \ref{lambdai} and by checking that, for all $\epsilon\in\{+,-\}^N$ such that $\epsilon_i=-$, we have $(1-\lambda_i)(m^\epsilon+m^{\epsilon'})=m^\epsilon$.
\end{proof}

\begin{lem}\label{lemme2}
For all $\epsilon\in\{+,-\}^N$ and all $i\in\{1,\dots,N\}$, we have $\phi_{x^\epsilon }(e_i^+)=0$ and $\phi_{x^\epsilon }(e_i^-)=a_i+\sum_{j>i,\,\epsilon_j=+}\alpha_{ij}^\epsilon a_j$ if $\epsilon_i=+$ (and conversely if $\epsilon_i=-$).
\end{lem}

\begin{proof}
Fix $\epsilon\in\{+,-\}^N$. The lemma follows from the three following steps.\\

Step 1. Let us first prove by induction that, for all $i\in\{1,\dots,N\}$, $$x^\epsilon_i=\sum_{\begin{array}{c}\scriptstyle i+1\leq h\leq N \\ \scriptstyle \epsilon_h=+\end{array}}\left(\sum_{k\geq 1}\sum_{\begin{array}{c}\scriptstyle i=i_0<i_1<\dots<i_k=h \\ \scriptstyle \forall x<k, \,\epsilon_{i_x}=-\end{array}}(-1)^{k+1}\prod_{x=0}^{k-1}\beta_{i_xi_{x+1}}\right)a_h+\left\{\begin{array}{ll} -a_i & \mbox{ if }\epsilon_i=+\\
 0 & \mbox{ if }\epsilon_i=-
 \end{array}\right.$$
Let $i\in\{1,\dots,N\}$.
Remark that if $\epsilon_i=+$, this equality is clearly true because for all $k\geq 1$ there exists no $i=i_0<i_1<\dots<i_k=h$ such that $\forall x<k, \,\epsilon_{i_x}=-$.
Remark also, for similar reason, that the sum from $h=i+1$ to $N$ can be replaced by the sum from $h=i$ to $N$ (always with the condition $\epsilon_h=+$).
Suppose now that $\epsilon_i=-$ and that for all $j>i$ the equality holds. Then
\begin{eqnarray*}
x^\epsilon_i & = & -\sum_{j>i}\beta_{ij}x^\epsilon_j \\
 & = &  -\sum_{j>i}\sum_{\begin{array}{c}\scriptstyle j\leq h\leq N \\\scriptstyle  \epsilon_h=+\end{array}}\left(\sum_{k\geq 1}\sum_{\begin{array}{c}\scriptstyle j=j_0<j_1<\dots<j_k=h \\ \scriptstyle  \forall x<k, \,\epsilon_{j_x}=-\end{array}}(-1)^{k+1}\beta_{ij}\prod_{x=0}^{k-1}\beta_{j_xj_{x+1}}\right)a_h+\sum_{\begin{array}{c}\scriptstyle j>i\\ \scriptstyle \epsilon_j=+\end{array}}\beta_{ij}a_j \\
 & = & \sum_{\begin{array}{c}\scriptstyle i+1\leq h\leq N \\ \scriptstyle \epsilon_h=+\end{array}}\left(\sum_{j=i+1}^h\sum_{k\geq 1}\sum_{\begin{array}{c}\scriptstyle j=j_0<j_1<\dots<j_k=h \\ \scriptstyle \forall x<k, \,\epsilon_{j_x}=-\end{array}}(-1)^{k+2}\beta_{ij}\prod_{x=0}^{k-1}\beta_{j_xj_{x+1}}\right)a_h+\sum_{\begin{array}{c}\scriptstyle j>i\\ \scriptstyle \epsilon_j=+\end{array}}\beta_{ij}a_j\\
  & = & \sum_{\begin{array}{c}\scriptstyle i+1\leq h\leq N \\ \scriptstyle \epsilon_h=+\end{array}}\left(\sum_{k\geq 2}\sum_{\begin{array}{c}\scriptstyle i=i_0<i_1<\dots<i_k=h \\ \scriptstyle \forall x<k, \,\epsilon_{i_x}=-\end{array}}(-1)^{k+2}\prod_{x=0}^{k-1}\beta_{i_xi_{x+1}}\right)a_h+\sum_{\begin{array}{c}\scriptstyle h>i\\ \scriptstyle \epsilon_h=+\end{array}}\beta_{ih}a_h.
\end{eqnarray*}
But for all $h\in\{i+1,\dots,h\}$, $\beta_{ih}$ equals $$\sum_{\begin{array}{c}\scriptstyle i=i_0<i_1<\dots<i_k=h \\ \scriptstyle \forall x<k, \,\epsilon_{i_x}=-\end{array}}(-1)^{k+2}\prod_{x=0}^{k-1}\beta_{i_xi_{x+1}}$$ when $k=1$, so that we obtain the wanted equation.\\

Step 2. For all $i\in\{1,\dots,N\}$ and $j\in\{i+1,\dots,N\}$ we have $$\alpha^\epsilon_{ij}=\sum_{k\geq 1}\sum_{\begin{array}{c}\scriptstyle i=i_0<i_1<\dots<i_k=j \\ \scriptstyle \forall x<k, \,\epsilon_{i_x}=-\end{array}}(-1)^{k+1}\prod_{x=0}^{k-1}\beta_{i_xi_{x+1}}.$$ The proof, by induction on $j$, of this formula is the same as in \cite[Lemma 3.5]{Pe05} and is left to the reader.\\

Step 3. Recall that $\phi_m(e_i^+)=m_i+a_i$ and that $\phi_m(e_i^-)=-m_i-\sum_{j>i}\beta_{ij}m_j$. Then, if $\epsilon_i=+$, we have $\phi_{x^\epsilon }(e_i^+)=0$ and $\phi_{x^\epsilon }(e_i^-)=-x^\epsilon_i-\sum_{j>i}\beta_{ij}x^\epsilon_j$. And, if $\epsilon_i=-$, we have $\phi_{x^\epsilon }(e_i^-)=0$ and $\phi_{x^\epsilon }(e_i^+)=a_i+x^\epsilon_i$. In fact, we only have to compute $\phi_{x^\epsilon }(e_i^+)$ in the case where  $\epsilon_i=-$, {\it i.e.} $a_i+x^\epsilon_i$. Indeed, if  $\epsilon_i=+$, define $\epsilon'\in\{+,-\}^N$ by $\epsilon'_j=\epsilon_j$ for all $j\neq i$ and $\epsilon'_i=-$. Then $\phi_{x^\epsilon }(e_i^-)=\phi_{x^{\epsilon'} }(e_i^+)$.
\end{proof}

We are now able to prove the vanishing theorem for divisors on the toric variety $\mathfrak{X}(0)$.

\begin{teo}\label{maintoric}
Let $\mathcal{D}=\sum_{i=1}^Na_i\mathcal{Z}_i$ be a divisor of $\mathfrak{X}$ and $\eta\in\{+,-,0\}^N$. 
Suppose that the coefficient $(a_i)_{i\in\{1,\dots,N\}}$ satisfy conditions $(C_i^{\eta_i})$ for all $i\in\{1,\dots,N\}$ such that $\eta_i\neq 0$.
Then $$H^i(\mathfrak{X}(0),\mathcal{D}(0))=0,\mbox{ for all }i<\sharp\{1\leq j \leq N\mid\eta_j=-\}\mbox{ and for all }i>N-\sharp\{1\leq j \leq N\mid\eta_j=+\}.$$
\end{teo}

\begin{proof}
Let $m\in X((\Bbbk^*)^N)$ such that $H^i(\mathfrak{X}(0),\mathcal{D}(0))_m$ is not zero for all $i\in\{1,\dots,N\}$. Then, by Corollary \ref{cortoriccohom} (i) and Lemma \ref{convexhull}, there exist non negative real numbers $m^\epsilon$ with $\epsilon\in\{+,-\}^N$ such that $\sum_{\epsilon\in\{+,-\}^N}m^\epsilon=1$ and $m=\sum_{\epsilon\in\{+,-\}^N}m^\epsilon x^\epsilon$. 

  Then, by Lemma \ref{lemme2}, $$\phi_m(e_i^+)=\sum_{\begin{array}{c}\scriptstyle \epsilon\in\{+,-\}^N\\ \scriptstyle \epsilon_i=-\end{array}}m^\epsilon C_i^\epsilon  \hspace{1cm}\mbox{ and }\hspace{1cm}  \phi_m(e_i^-)=\sum_{\begin{array}{c}\scriptstyle \epsilon\in\{+,-\}^N\\ \scriptstyle \epsilon_i=+\end{array}}m^\epsilon C_i^\epsilon.$$

Then, if Condition $C_i^-$ is satisfied, we have $\phi_m(e_i^+)$ and $\phi_m(e_i^-)$ are both negative. And if Condition $C_i^+$ is satisfied and if  the integers $\phi_m(e_i^+)$ and $\phi_m(e_i^-)$ are not both non-negative, then one of them equals $-1$ (say for example $\phi_m(e_i^+)$). It means that for all $\epsilon \in\{+,-\}^N$ such that $\epsilon_i=+$, we have $m^\epsilon=0$. Then $\phi_m(e_i^-)=0$ that is not possible by hypothesis on $m$ and Corollary \ref{cortoriccohom} (i). 

We conclude the proof by Corollary \ref{cortoriccohom} (ii).
\end{proof}

\begin{ex}
If $G=\operatorname{SL(3)}$ and $\tilde{w}=s_{\alpha_1}s_{\alpha_2}$, the vanishings of the cohomology of the divisor $\mathcal{D}=a_1\mathcal{Z}_1+a_2\mathcal{Z}_2$ obtained by Theorem \ref{maintoric} is reprensented in the following picture.
\begin{center}
\includegraphics[width=12cm]{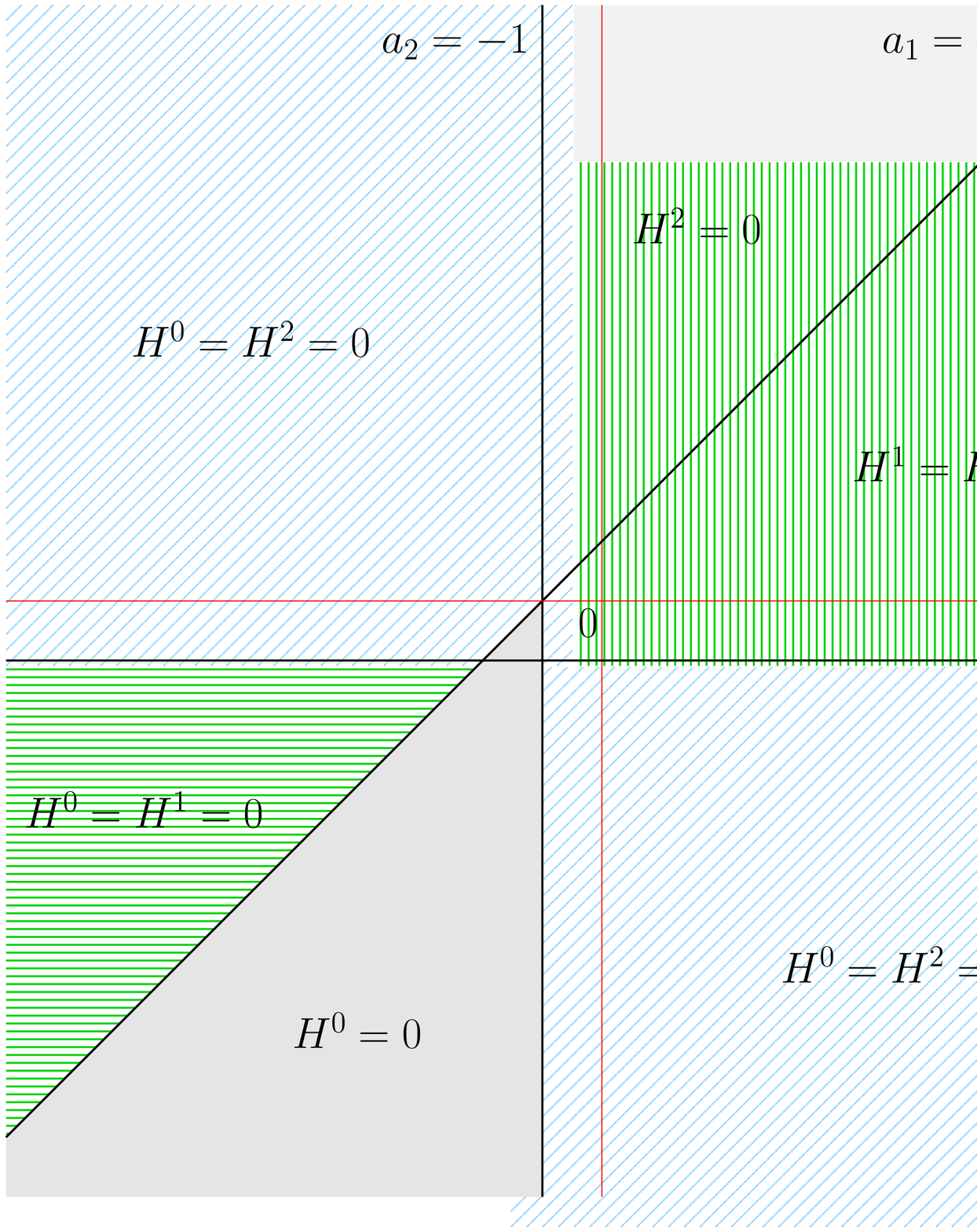}
\end{center}
\end{ex}

Let us now discuss, with a more general example, what sort of vanishings Theorem \ref{main} gives.

\begin{ex}\label{exemple}
Let $G=\operatorname{SL(4)}$ and $\tilde{w}=s_{\alpha_2}s_{\alpha_1}s_{\alpha_3}s_{\alpha_2}$ (with natural notation).
Let $D=\sum_{i=1}^4a_iZ_i$ be a divisor of $X(\tilde{w})$. Then, all the integers $C_i^\epsilon$ we obtain are the followings:
$$\begin{array}{cc}
i=4 & a_4\\
i=3 & a_3,\,a_3-a_4\\
i=2 & a_2,\,a_2-a_4\\
i=1 & a_1,\,a_1-a_2,\,a_1-a_3,\,a_1-a_2-a_3,\,a_1-a_2+a_4,\,a_1-a_3+a_4,\,a_1-a_2-a_3+2a_4.
\end{array}$$
In particular, Conditions $(C_i^+)_{i\in\{1,2,3,4\}}$ are equivalent to $a_4\geq -1$, $a_3\geq a_4-1$, $a_2\geq a_4-1$ and $a_1\geq a_2+a_3-1$. In that case, Theorem \ref{main} tells us that the cohomology of $D$ vanishes in non zero degree. But this fact can already be deduced by \cite[Theorem 7.4]{LT04}. Actually, the theorem of N.~Lautitzen and J.F.~Thomsen gives us the vanishing of the cohomology of $D$ in non zero degree exactly for all $D$ such that only if $a_4\geq -1$, $a_3\geq \max(a_4-1,-1)$, $a_2\geq \max(a_4-1,-1)$ and $a_1\geq \max(a_2+a_3-a_4-1,-1)$. 

Let us consider $D=2Z_1+2Z_2+2Z_3+2Z_4$, by the latter assertion the cohomology of $D$ in non zero degree vanishes. But one can compute that the cohomology of the corresponding divisor on $\mathfrak{X}(0)$ is not trivial in degree~1 (indeed, we have for example $H^1(\mathfrak{X}(0),\mathcal{D}(0))_m=\Bbbk$ when $m=\frac{1}{2}(x^{(-,+,+,+)}+x^{(-,+,+,-)})=(0,-2,-2,-3)$). 

Theorem \ref{main} is not as powerful as the results of N.~Lautitzen and J.F.~Thomsen for ``positive'' divisors (or also for ``negative'' divisors). But for all other divisors it gives many new vanishings results.

For example, if $a_4\geq 0$, $a_3\geq a_4$, $a_2<0$, Theorem \ref{main} gives the vanishing of the cohomology of $D$ in degree 0, 3 and 4.
\end{ex}

\begin{rems}
Theorem \ref{main} is easy to apply to a given divisor of a Bott-Sameslon variety. Indeed, we made a program that takes a triple $(A,\tilde{w},Z)$ consisting of a Cartan matrix $A$, an expression $\tilde{w}$ and a divisor $Z$ of $X(\tilde{w})$, and that computes the vanishing results in the cohomology of $Z$ given by Theorem \ref{main} (contact the author for more detail).

We can also obtain vanishing results in the cohomology of line bundles on Schubert varieties. These results are also computable. Then, we remark that, as for Bott-Samelson varieties, we do not recover all the already-known vanishing results on ``positive'' line bundles, but it gives new results for more general line bundles. And we also remark that the result we obtain depends on the choosen reduced expression of the element of the Weil group associated to the Schubert variety.
\end{rems}

\bigskip\noindent

\medskip\noindent
Boris {\sc Pasquier}, Hausdorff Center for Mathematics,
Universit{\"a}t Bonn, Landwirtschaftskammer (Neubau)
Endenicher Allee 60, 53115 Bonn, Germany.

\noindent {\it email}: \texttt{boris.pasquier@hcm.uni-bonn.de}

\end{document}